\documentclass[12pt,reqno]{amsart}
\usepackage[english]{babel}
\usepackage{amsmath,amsfonts,amssymb,amsthm,amscd,latexsym}
\usepackage[T1]{fontenc}
\usepackage[utf8]{inputenc}

\usepackage{tikz}
\usetikzlibrary{arrows}
\usepackage{color}
\usepackage{cite}

\usepackage{multirow}

\DeclareMathOperator{\Aut}{Aut}

\DeclareMathOperator{\Hom}{Hom}
\DeclareMathOperator{\id}{id}
\newtheorem{thm}{Theorem}[section]
\newtheorem{cor}[thm]{Corollary}
\newtheorem{lm}[thm]{Lemma}
\newtheorem{prop}[thm]{Proposition}
\newtheorem{defn}[thm]{Definition}
\newtheorem{rem}[thm]{Remark}
\newtheorem{exam}[thm]{Example}

\numberwithin{equation}{section}

\begin{document}


\author[Kudaybergenov]{Karimbergen Kudaybergenov}
\email{karim2006@mail.ru}
\address{[K.\ Kudaybergenov] Department of Mathematics, Karakalpak State University, Ch. Abdirov 1, Nukus 230113, Uzbekistan}

\author[Ladra]{Manuel Ladra}
\email{manuel.ladra@usc.es}
\address{[M.\ Ladra] Department of Mathematics, Institute of Mathematics, University of Santiago de Compostela, 15782, Spain}

\author[Omirov]{Bakhrom Omirov}
\email{omirovb@mail.ru}
\address{[B.\ Omirov] National University of Uzbekistan,  4, University street, 100174, Tashkent,   Uzbekistan}

\title[On Levi--Malcev theorem for Leibniz algebras]{On Levi--Malcev theorem for Leibniz algebras}

\begin{abstract}
The present paper is devoted to provide conditions for the Levi--Malcev theorem to hold or not to hold
(i.e. for two Levi subalgebras to be or not conjugate by an inner automorphism) in the context of finite-dimensional Leibniz algebras over a field of characteristic zero.
 Particularly, in the case of the field $\mathbb{C}$ of complex numbers, we consider all possible cases in which Levi subalgebras are conjugate and not conjugate.
\end{abstract}


\maketitle

\section{Introduction}
 The Levi decomposition, proved by E. E. Levi \cite{Levi}, states that any finite-dimensional real Lie algebra $\mathcal{G}$
 is the semidirect sum of the solvable radical and a semisimple subalgebra (see also \cite{Jac} for the case of characteristic zero),
  and it plays an important role in  Lie theory and representation theory. The semisimple subalgebra of the decomposition is called a Levi subalgebra of  $\mathcal{G}$.
   Moreover, Malcev showed that any two Levi subalgebras are conjugate by an inner automorphism (Levi--Malcev theorem) \cite{Malcev}.

There exists an analogue of Levi--Malcev decomposition for simply connected Lie groups and simply connected algebraic groups over a field of characteristic zero \cite{Onishchik}.
 There again, there is no analogue of the Levi decomposition for most infinite-dimensional Lie algebras; for example affine Lie algebras have a radical consisting of their center,
  but they cannot be written as a semidirect sum of the center and another Lie algebra. There is not a Levi decomposition for finite-dimensional algebras over fields of positive characteristic either.

A comprehensive study of the Lie algebra theory resulted in a number of beautiful results and generalizations. In particular, Loday introduced in~\cite{Lod} a non skew-symmetric analogue of  a Lie algebra, called Leibniz algebra.

A simple, but yet productive property from Lie theory, namely the fact that the right multiplication operator  on an element of the algebra is a derivation,
 can also be taken as a defining property for a Leibniz algebra. In the last years, Leibniz algebras have been under active research; among the numerous papers devoted to this subect, we can find
  some (co)homology and deformations properties, results on various types of decompositions, structure of solvable and nilpotent Leibniz algebras and
   classifications of some classes of graded nilpotent Leibniz algebras.
    Also, many results of theory of Lie algebras have been extended to the Leibniz algebras case. For instance, the classical results on Cartan subalgebras,
    Levi decomposition, Killing form, Engel's theorem, properties of solvable algebras with a given nilradical and others from the theory of Lie algebras are also true for Leibniz algebras.

Recently, D. Barnes proved an analogue of Levi's theorem for the case of Leibniz algebras \cite{Barnes}; namely, a Leibniz algebra $\mathcal{L}$ is decomposed into a semidirect sum of its solvable radical and a semisimple Lie subalgebra, $\mathcal{L}=\mathcal{S}\dot{+} \mathcal{R}$.
  He also presents an example in which two semisimple Lie subalgebras corresponding to different decompositions are not conjugate by an inner automorphism.
It is our objective in this paper to investigate when two semisimple Lie subalgebras in this context are conjugate via (inner) automorphisms.

The organization of this paper is as follows. In Section~\ref{S:prel}  we give some preliminary well-known results and
some technical lemmas about Lie and Leibniz algebras. In Section~\ref{S:conjugacy} we give,  in characteristic zero,  some sufficient conditions  for two Levi subalgebras to be conjugate or not.
Namely, if $\mathcal{L}=\mathcal{S}\dot{+} \mathcal{R}$ is a Levi decomposition of a  Leibniz algebra $\mathcal{L}$,  $\mathcal{I}$ is the ideal generated by its squares and
$\mathcal{J}$ is the maximal submodule of
$\mathcal{I}$ such that $\Hom_\mathcal{S}(\mathcal{S},
\mathcal{I})\equiv \Hom_\mathcal{S}(\mathcal{S},
\mathcal{J})$, then we provide, depending on the value of $\mathcal{J}$, sufficient conditions  for two Levi subalgebras to be or not to be conjugate. In particular, if $\mathcal{J}=0$ then two Levi subalgebras are conjugate by an inner automorphism (Proposition~\ref{homzero}). If  $\mathcal{J} \neq 0$  and $[\mathcal{J}, \mathcal{R}] =0 $, then two Levi subalgebras are conjugate by non-inner automorphisms (Proposition~\ref{conauto}). If  $\mathcal{J} \neq 0$  and depending on the values $[\mathcal{S}, \mathcal{R}]$, $[\mathcal{S}, \mathcal{R}]$, equal to zero or not, with some additional assumptions, we obtain that they are not conjugate (Propositions~\ref{notconj} and \ref{exam02}). Finally, in  Section~\ref{S:complex}, in the case of the field $\mathbb{C}$ of complex numbers, we provide a  necessary and sufficient condition for  two Levi subalgebras to be conjugate by inner automorphisms (Theorem~\ref{T:equiv}).

Throughout the paper, linear spaces and algebras are finite-dimensional over a field of characteristic zero if it is not specified.

\section{Preliminaries}\label{S:prel}

In this section we give necessary definitions and results.

\begin{defn} An algebra $(\mathcal{L},[\cdot,\cdot])$ over a field $\mathbb{F}$ is called a
Leibniz algebra if it satisfies the property
\[ [x,[y,z]]=[[x,y],z] - [[x,z],y],  \  \text{for all} \  x,y,z \in \mathcal{L},\]
which is called Leibniz identity.
\end{defn}

The Leibniz identity is a generalization of the Jacobi identity, since under the condition of anti-symmetricity of the product ``[$\cdot\, ,\,\cdot$]'' this identity changes to the Jacobi identity.

For a Leibniz algebra $\mathcal{L}$, the subspace generated by its squares $\mathcal{I}= \left<\{[x,x]:  x\in \mathcal{L}\}\right>$ becomes an ideal due to Leibniz identity, and
the quotient $\widetilde{\mathcal{L}}=\mathcal{L}/\mathcal{I}$ is
a Lie algebra called the liezation of $\mathcal{L}$. Since we are focused in Leibniz algebras
which are not Lie algebras, we will always assume that
$\mathcal{I}\ne0$.

The notions of nilradical and solvable radical (denoted further by $\mathcal{N}$ and $\mathcal{R}$, respectively) are defined similarly to the case of Lie algebras (see \cite{Ayup}).
 Since in a Leibniz algebra $\mathcal{L}$ the equality $[x,[y,y]]=0$ holds for any $x,y\in \mathcal{L}$, we derive $[\mathcal{L},\mathcal{I}]=0$. Therefore,
\[\mathcal{I}\subseteq \mathcal{N} \subseteq \mathcal{R}.\]

We recall an analogue of Levi--Malcev theorem for Leibniz algebras.

\begin{thm}[\cite{Barnes}]   Let $\mathcal{L}$ be a finite-dimensional Leibniz algebra over a
field of characteristic zero and let $\mathcal{R}$ be its solvable
radical. Then there exists a semisimple Lie subalgebra
$\mathcal{S}$ of $\mathcal{L}$ such that
$\mathcal{L}=\mathcal{S}\dot{+}\mathcal{R}$.
\end{thm}

The subalgebra $\mathcal{S}$ of the above theorem, similarly to Lie algebras theory, is called a \emph{Levi subalgebra} of the Leibniz algebra $\mathcal{L}$.

Let \(\mathcal{N}\) be the nilradical of a Leibniz algebra \(\mathcal{L}\).
Let \(a \in \mathcal{N}\) and $R_a$  the  right multiplication operator on $a$. We set
\[\exp(R_a)(x)=x+R_a(x)+\frac{R_a^2(x)}{2!}+\dots+\frac{R_a^n(x)}{n!}.\]

Then \(\exp(R_a)\) is an automorphism of \(\mathcal{L}\).

An automorphism obtained by composition of such type automorphisms is called an \emph{inner automorphism}.

Below, we present an adapted version to our further using  of the
well-known Schur's Lemma (\cite[p.57, Corollary 3]{Zhel}).

\begin{lm}  \label{schur}
 Let $\mathcal{G}$ be a complex Lie algebra,   $\mathcal{V}=\mathcal{V}_1\oplus\dots \oplus\mathcal{V}_m$ and $\mathcal{W}=\mathcal{W}_1\oplus\dots
\oplus\mathcal{W}_n$  completely reducible
$\mathcal{G}$-modules, where \(\mathcal{V}_1, \dots,
\mathcal{V}_n, \mathcal{W}_1,\dots ,\mathcal{W}_n\) are
irreducible modules. Then any $\mathcal{G}$-module homomorphism
$\varphi \colon \mathcal{V} \rightarrow \mathcal{W}$ can be represented
as
\[
\varphi=\sum\limits_{i=1}^m\sum\limits_{j=1}^n \lambda_{i
j}\varphi_{i j},
\]
where the operators \(\varphi_{i j} \colon \mathcal{V}_i\to\mathcal{W}_j
\)  are fixed $\mathcal{G}$-module homomorphisms  and  \(\lambda_{i j}\) are
complex numbers. Furthermore,   \(\varphi_{i j}\ne0\)  if and only
if  \(\varphi_{i j}\) is an  isomorphism. 
\end{lm}

\section{On conjugacy of Levi subalgebras of Leibniz algebras} \label{S:conjugacy}

Firstly we show that ideals $\mathcal{I},\ \mathcal{N}$ and $\mathcal{R}$ are invariant under an automorphism of a Leibniz algebra $\mathcal{L}$.

\begin{prop} \label{propideals}
\[\varphi(\mathcal{I})=\mathcal{I}, \quad \varphi(\mathcal{N})=\mathcal{N}, \quad \varphi(\mathcal{R})=\mathcal{R}, \ \text{where} \ \varphi\in \Aut(\mathcal{L}).\]
\end{prop}
\begin{proof} Let $\varphi \in \Aut(\mathcal{L})$. Then for any $[x,x]\in \mathcal{I}$ we have
$\varphi([x,x])=[\varphi(x),\varphi(x)]\in \mathcal{I}$. This implies $\varphi (\mathcal{I})\subseteq \mathcal{I}$. Now for $x=\sum \alpha_i[x_i,x_i]\in \mathcal{I}$
we have the existence of $y_i\in \mathcal{I}$ such that $\varphi(y_i)=x_i$. Due to
$x=\sum \alpha_i[x_i,x_i]=\sum \alpha_i[\varphi(y_i),\varphi(y_i)]=\varphi(\sum \alpha_i[y_i, y_i])$,
we conclude $x=\varphi(y)$, where $y=\sum \alpha_i[y_i, y_i]\in \mathcal{I}$. Thus, $\varphi(\mathcal{I})=\mathcal{I}$.

For any $k\in \mathbb{N}$ we have
\[\varphi(\mathcal{N})^k=\varphi(\mathcal{N}^k), \qquad \varphi(\mathcal{R})^{[k]}=\varphi(\mathcal{R}^{[k]}),\]
where \quad $\mathcal{N}^1=\mathcal{N}, \ \mathcal{N}^{k+1}=[\mathcal{N}^k,\mathcal{N}],   \qquad
\mathcal{R}^{[1]}=\mathcal{R}, \ \mathcal{R}^{[k+1]}=[\mathcal{R}^{[k]},\mathcal{R}^{[k]}], \quad k \geq 1$,

\noindent are  the \emph{lower central} and the \emph{derived series} of $\mathcal{N}$ and  $\mathcal{R}$, respectively.
Consequently, $\varphi(\mathcal{N})$ and $\varphi(\mathcal{R})$ are nilpotent and solvable ideals, respectively; that is,
 $\varphi(\mathcal{N})\subseteq \mathcal{N}$ and $\varphi(\mathcal{R})\subseteq \mathcal{R}$. Applying dimensional arguments we complete the proof of the proposition.
\end{proof}

Let $\mathcal{L}=\mathcal{S}\dot{+} \mathcal{R}$ be a Leibniz algebra and let \(\varphi\) be an automorphism of \(\mathcal{L}\).
 Due to Proposition~\ref{propideals} we get $\varphi = \varphi_{\mathcal{S}, \mathcal{S}} +\varphi_{\mathcal{S},\mathcal{R}}+\varphi_{\mathcal{R},\mathcal{R}}$.

Denote by \(\Hom_\mathcal{S}(\mathcal{S}, \mathcal{I})\)
the set of all \(\mathcal{S}\)-module homomorphisms from
\(\mathcal{S}\) into \(\mathcal{I}\).

We set
\begin{equation}\label{test}
\mathcal{S}_\theta=\left\{x+\theta(x): x\in \mathcal{S}, \theta\in \Hom_\mathcal{S}(\mathcal{S},
\mathcal{I})\right\}.
\end{equation}

Since
\[[x+\theta(x), y+\theta(y)]=[x, y]+[\theta(x), y]=[x,
y]+\theta([x, y]),\] it follows that the mapping \(x \mapsto
x+\theta(x)\in \mathcal{S}_\theta\), $x\in \mathcal{S}$, is an automorphism from
\(\mathcal{S}\) onto \(\mathcal{S}_\theta\).

\begin{rem} Since the proofs of the results for the automorphisms of the form $\exp(R_a)$ can easily extended to the case of inner automorphisms,
we shall further consider only automorphisms of the form $\exp(R_a)$.
\end{rem}

\begin{thm} \label{thmmain}
  Let $\mathcal{S}$ and $\mathcal{S}_1$ be two Levi subalgebras of a Leibniz algebra $\mathcal{L}$.
Then there exist \(\tau \in \Hom_\mathcal{S}(\mathcal{S}, \mathcal{I})\) and an
element \(a\in \mathcal{N}\) such that \(\exp(R_a)(\mathcal{S}_\tau)=\mathcal{S}_1\).
\end{thm}

\begin{proof} Let \(\widetilde{\mathcal{L}}=\mathcal{S}\dot{+}
\widetilde{\mathcal{R}},\) where
\(\widetilde{\mathcal{R}}=\mathcal{R}/\mathcal{I}.\) Due to
Levi--Malcev theorem for Lie algebras (see \cite{Malcev}), there
exists an element \(\widetilde{a}\in
\widetilde{\mathcal{N}}=\mathcal{N}/\mathcal{I}\) such that
\(\exp(R_{\widetilde{a}})(\mathcal{S})=\mathcal{S}_1\). For an
element \(a\) from the class \(\widetilde{a}\), take an
automorphism \(\exp(R_a)\) of \(\mathcal{L}\).

Let \(\pi \colon \mathcal{L}\to \widetilde{\mathcal{L}}\) be the quotient
map. Since any automorphism maps \(\mathcal{I}\) into itself, the
quotient map \(\widetilde{\exp(R_a)}\) is well defined, i.e.
\[
\widetilde{\exp(R_a)}(\pi(x))=\pi\left(\exp(R_a)(x)\right),\, x\in
\mathcal{L}.
\]
Then \(\widetilde{\exp(R_a)}\) is an automorphism of
\(\widetilde{\mathcal{L}}\). For \(x\in \mathcal{S}\) we have
\begin{align*}
\widetilde{\exp(R_a)}(x) & =
\widetilde{\exp(R_a)}(\pi(x))=\pi\left(\exp(R_a)(x)\right)=\pi\left(\sum\limits_{k=0}^n\frac{R_a^k(x)}{k!}\right)\\
{} & = \sum\limits_{k=0}^n\frac{R_{\pi(a)}^k(\pi(x))}{k!}=\sum\limits_{k=0}^n\frac{R_{\widetilde{a}}^k(x)}{k!}=
\exp(R_{\widetilde{a}})(x).
\end{align*}
Therefore,
$\widetilde{\exp(R_a)}\big|_{\mathcal{S}}=\left.\exp(R_{\widetilde{a}})\right|_{\mathcal{S}}$.

Thus, for every \(x\in \mathcal{S}\) there exists an element
\(\tau(x)\in \mathcal{I}\) such that
\[
\exp(R_a)^{-1}\left(\exp(R_{\widetilde{a}})(x)\right)=x+\tau(x).
\]

For \(x,y \in \mathcal{S}\) we consider
\begin{align*}
[x, y]+\tau([x, y]) & =
\exp(R_a)^{-1}\left(\exp(R_{\widetilde{a}})([x,
y])\right)\\
{}&= \left[\exp(R_a)^{-1}\left(\exp(R_{\widetilde{a}})(x)\right),
\exp(R_a)^{-1}\left(\exp(R_{\widetilde{a}})(y)\right)\right]\\
{} &= \left[x+\tau(x), y+\tau(y)\right]=[x,y]+[\tau(x), y].
\end{align*}

Consequently, \(\tau \in \Hom_\mathcal{S}(\mathcal{S},\mathcal{I})\).
Moreover,
\[
\exp(R_{\widetilde{a}})(x)=\exp(R_a)(x+\tau(x)) \ \text{for all} \ x\in \mathcal{S}
\]
implies $\mathcal{S}_1=\exp(R_a)(\mathcal{S}_{\tau})$.
\end{proof}

The following result follows from Theorem~\ref{thmmain} and it asserts that the question on the conjugation of Levi subalgebras is reduced
to verifying of the conjugation of Levi subalgebras \(\mathcal{S}\) and \(\mathcal{S}_\theta\), where the Levi subalgebra \(\mathcal{S}_\theta\) is an algebra of the
form~\eqref{test}.

\begin{cor} \label{cormain}  A Levi subalgebra \(\mathcal{S}\) is unique up to conjugation by
an inner automorphism of \(\mathcal{L}\) if and only if for every \(\tau\in\Hom_\mathcal{S}(\mathcal{S}, \mathcal{I})\)
 there exists an element \(b\in \mathcal{N}\) such that \(\left.\exp(R_b) \right|_\mathcal{S}=\id_\mathcal{S}+\tau\).
\end{cor}

In the next two propositions we give sufficient conditions to conjugations of Levi subalgebras.

\begin{prop} \label{homzero}
  Let $\mathcal{S}$ be a Levi subalgebra of a Leibniz algebra $\mathcal{L}$ and \(\Hom_\mathcal{S}(\mathcal{S}, \mathcal{I})=\{0\}\).
 Then \(\mathcal{S}\) is unique up to conjugation by an inner automorphism of \(\mathcal{L}\).
\end{prop}
\begin{proof} Taking into account the condition \(\Hom_\mathcal{S}(\mathcal{S},\mathcal{I})=\{0\}\) in
Theorem~\ref{thmmain}, we obtain \(\mathcal{S}_\tau=\mathcal{S}\) and therefore \(\exp(R_a)(\mathcal{S})=\mathcal{S}_1\).
\end{proof}

\begin{rem}
It should be noted that the equality $[\mathcal{I},\mathcal{S}]=0$ implies $\Hom_\mathcal{S}(\mathcal{S},\mathcal{I})=\{0\}$.
 Therefore, due to Proposition~\ref{homzero} we conclude that Levi subalgebras of a Leibniz algebra $\mathcal{L}=\mathcal{S}\dot{+}\mathcal{R}$
 satisfying the condition $[\mathcal{I},\mathcal{S}]=\{0\}$ are conjugate via an inner automorphism.

\end{rem}

Thanks to Proposition~\ref{homzero}, henceforth we will consider only the case when \(\Hom_\mathcal{S}(\mathcal{S}, \mathcal{I})\neq \{0\}\).

Let \(\mathcal{L}=\mathcal{S}\dot{+}\mathcal{R}\) be a Leibniz
algebra.  A Levi subalgebra \(\mathcal{S}\) and
\(\mathcal{S}\)-module \(\mathcal{I}\) can be uniquely represented
as
\[ \mathcal{S}=\mathcal{G}\oplus \mathcal{Q} \quad  \text{and} \quad  \mathcal{I}=\mathcal{J}\oplus \mathcal{K}\]
such that
\[
\Hom_\mathcal{S}(\mathcal{S},
\mathcal{I})\equiv \Hom_\mathcal{S}(\mathcal{G},
\mathcal{J}) \quad  \text{and} \quad   \Hom_\mathcal{S}(\mathcal{S},
\mathcal{K})=\{0\}.
\]

Let \(\mathcal{L}=\mathcal{S}\dot{+}\mathcal{R}\) be a Levi decomposition of a Leibniz
algebra $\mathcal{L}$ such that \([\mathcal{J}, \mathcal{R}]=0\).

For \(\theta\in \Hom_\mathcal{S}(\mathcal{S}, \mathcal{I})\) we
set
\[
\delta_\theta(x_\mathcal{S}+x_\mathcal{R})=\theta(x_\mathcal{S}),\quad
x=x_\mathcal{S}+x_\mathcal{R}\in \mathcal{S}\dot{+}\mathcal{R}.
\]
Then \(\delta_\theta\) is a derivation of \(\mathcal{L}\). Indeed,
\begin{align*}
\left[\delta_\theta(x), y\right]+\left[x, \delta_\theta(y)\right]
& =  [\theta(x_\mathcal{S}), y_\mathcal{S}+y_\mathcal{R}]+[x,
\theta(y_\mathcal{S})]\\
{}&= [\theta(x_\mathcal{S}), y_\mathcal{S}]+[\theta(x_\mathcal{S}),
y_\mathcal{R}]=[\theta(x_\mathcal{S}), y_\mathcal{S}] \\
{}& = \theta([x_\mathcal{S}, y_\mathcal{S}])=\delta_\theta([x,
y]).
\end{align*}

It is clear that \(\delta_\theta^2=0\). Hence,
$\exp(\delta_\theta)(x)=x+\delta_\theta(x),\, x\in \mathcal{L}$
and $\exp(\delta_\theta)(x)=x+\theta(x),\, x\in \mathcal{S}$.
Thus, we obtain \[\exp(\delta_\theta)(\mathcal{S})=\mathcal{S}_\theta.\]

Consider a nilpotent derivation of the form:
\begin{equation}\label{nilder}
D=R_a+\delta_\theta, \quad  a\in \mathcal{N}, \quad  \theta\in \Hom_\mathcal{S}(\mathcal{S}, \mathcal{I}).
\end{equation}

\begin{prop} \label{conauto}
  Let $\mathcal{L}=\mathcal{S}\dot{+}\mathcal{R}$ be a Levi
decomposition of a Leibniz algebra $\mathcal{L}$ such that \([\mathcal{J}, \mathcal{R}]=\{0\}\).
Then any two Levi subalgebras $\mathcal{S}$ and $\mathcal{S}_1$ are conjugate by \(\exp(D)\), where \(D\) has the form \eqref{nilder}.
\end{prop}
\begin{proof}
By Theorem~\ref{thmmain} there exist \(\tau \in
\Hom_\mathcal{S}(\mathcal{S}, \mathcal{I})\) and an
element \(a\in \mathcal{N}\) such that
\(\exp(R_a)(\mathcal{S}_\tau)=\mathcal{S}_1\). Since
\([\mathcal{J}, \mathcal{R}]=0\), it follows that \(\delta_\tau\)
is a derivation with $\delta^2_\tau=0$. For a derivation  \(D=R_a+\delta_\tau\), we have
\begin{align*}
\exp(D)(\mathcal{S}) & =
\exp(R_a+\delta_\tau)(\mathcal{S})=(\exp(R_a)\circ
\exp(\delta_\tau))(\mathcal{S})\\
{}&=
\exp(R_a)(\exp(\delta_\tau)(\mathcal{S}))=\exp(R_a)(\mathcal{S}_\tau)=\mathcal{S}_1.
\end{align*}
The proof is complete.
\end{proof}

 Note that the example presented in the paper \cite{Barnes} satisfies the conditions of Proposition~\ref{conauto}.

In the next two propositions consider the cases of Leibniz algebras with Levi subalgebras which are not conjugate via any automorphism.

\begin{prop} \label{notconj} Let \(\mathcal{L}=\mathcal{S}\dot{+}
\mathcal{R}\) be a Levi decomposition with conditions
\([\mathcal{S}, \mathcal{R}]=\{0\}\) and \([\mathcal{J},
\mathcal{R}]\neq \{0\}\). Then there exists a Levi subalgebra
$\mathcal{S}_1\dot{+} \mathcal{R}$ such that the algebras
\(\mathcal{S}\) and \(\mathcal{S}_1\) are not conjugate via any
automorphism.
\end{prop}

\begin{proof}
Because of \([\mathcal{J}, \mathcal{R}]\neq \{0\}\) we have the
existence of elements \(i\in \mathcal{J}\) and \(y\in
\mathcal{R}\) such \([i, y]\neq 0\).

Consider an $\mathcal{S}$-module homomorphism $\theta$ from $\mathcal{S}$ to $\mathcal{J}$
such that there exists $x\in \mathcal{S}$ with the condition $\theta(x)=i$.

We set
\[ \mathcal{S}_1=\left\{x+\theta(x): x\in \mathcal{S}\right\}.\]

Let assume the existence of an automorphism $\varphi=\varphi_{\mathcal{S},\mathcal{S}}+\varphi_{\mathcal{S},\mathcal{R}}+\varphi_{\mathcal{R},\mathcal{R}} \in \Aut(\mathcal{L})$
 such that $\varphi(\mathcal{S})=\mathcal{S}_1$.

For \(x\in \mathcal{S}\) we have $\varphi_{\mathcal{S},\mathcal{S}}(x)+\varphi_{\mathcal{S},\mathcal{R}}(x)=\varphi(x)\in
\mathcal{S}_1\subset \mathcal{S}+\mathcal{J}$.

Hence \(\varphi_{\mathcal{S},\mathcal{R}}(x)\in \mathcal{J}\) and \(\varphi_{\mathcal{S},\mathcal{R}}=\varphi_{\mathcal{S},\mathcal{J}}\).

Since $\varphi(\mathcal{R})=\mathcal{R}$, then there exists $z\in \mathcal{R}$ such that $\varphi_{\mathcal{R},\mathcal{R}}(z)=y$.

Taking into account in the following chain of equalities \([\mathcal{S}, \mathcal{R}]=0\),
\begin{align*}
0 &  =  \varphi([x, z])=[\varphi(x),
\varphi(z)]=[\varphi_{\mathcal{S},\mathcal{S}}(x)+\varphi_{\mathcal{S},\mathcal{J}}(x),
\varphi_{\mathcal{R},\mathcal{R}}(z)]\\
{} &= [\varphi_{\mathcal{S},\mathcal{J}}(x), \varphi_{\mathcal{R},\mathcal{R}}(z)]=[i, y]\neq 0.
\end{align*}
From this contradiction we obtain that   there is no any
automorphism of \(\mathcal{L}\) which maps \(\mathcal{S}\) onto
\(\mathcal{S}_1\). The proof is complete.
\end{proof}

Let \(\mathcal{G}=\mathcal{S}\dot{+}\mathcal{M}\) be a Levi decomposition of a Lie algebra \(\mathcal{G}\) satisfying the following conditions:
\[\mathcal{M}=\mathcal{N}+\left< p \right>, \quad \mathcal{N}=[\mathcal{M}, \mathcal{M}], \quad [\mathcal{S}, p]=\{0\}, \quad [\mathcal{S}, x]\neq 0, \ \text{for any} \ 0\neq x \in \mathcal{N}.\]

We set
\[\mathcal{I}=\mathcal{S} \quad \text{and} \quad\theta=\left. \id\right|_{\mathcal{S}} \colon \mathcal{S} \to \mathcal{I}.\]

On the space \(\mathcal{L}=\mathcal{S}\dot{+}\mathcal{M}\dot{+}\mathcal{I}\),
we define the multiplication table  in the following way:

products on \(\mathcal{S} \dot{+} \mathcal{M}\) remain unchanged (similar as in Lie algebra $\mathcal{G}$),
\[[\mathcal{L},\mathcal{I}]=[\mathcal{I}, \mathcal{N}]=0, \quad
[\theta(x), y]=\theta([x, y]) \ \text{for} \ x, y \in \mathcal{S}, \quad [i, p]=i, \ \text{for} \ i\in \mathcal{I}.\]

Straightforward verification of Leibniz identity shows that $\mathcal{L}$ is a Leibniz algebra.

Note that \(\mathcal{S}\) and \(\mathcal{I}\) are isomorphic
\(\mathcal{S}\)-modules via  the isomorphism \(\theta\).

\begin{prop} \label{exam02}
The above mentioned Leibniz algebra \(\mathcal{L}\)
admits two Levi subalgebras which are non conjugate by any automorphism of $\mathcal{L}$.
\end{prop}
\begin{proof} Consider the Levi subalgebra $\mathcal{S}_\theta=\{x+\theta(x): x\in \mathcal{S}\}$.

Suppose that  there is $\varphi=\varphi_{\mathcal{S},\mathcal{S}}+\varphi_{\mathcal{S},\mathcal{R}}+\varphi_{\mathcal{R},\mathcal{R}}\in \Aut\mathcal{L}$ such that $\varphi(\mathcal{S})=\mathcal{S}_\theta$.

Similar as in the proof of Proposition~\ref{notconj} we obtain
\(\varphi_{\mathcal{S},\mathcal{R}}=\varphi_{\mathcal{S},\mathcal{I}}\neq 0\).

Taking into account \(\varphi(\mathcal{N})=\mathcal{N}\) we obtain that \(\varphi(p)=b+c\), where \(b=\alpha p\neq 0\) and \(c\in \mathcal{N}\).

For an element $0\neq x\in \mathcal{S}$ we have
\begin{align*}
0 &  =  \varphi(0)=\varphi([x, p])=[\varphi(x),
\varphi(p)]=[\varphi_{\mathcal{S},\mathcal{S}}(x),
\varphi_{\mathcal{R},\mathcal{R}}(p)]+[\varphi_{\mathcal{S},\mathcal{I}}(x),
\varphi_{\mathcal{R},\mathcal{R}}(p)]\\
{}&= [\varphi_{\mathcal{S},\mathcal{S}}(x), b+c]+[\varphi_{\mathcal{S},\mathcal{I}}(x), b+c]=[\varphi_{\mathcal{S},\mathcal{S}}(x), c]+[\varphi_{\mathcal{S},\mathcal{I}}(x), b]
=[\varphi_{\mathcal{S},\mathcal{S}}(x), c]+\alpha\varphi_{\mathcal{S},\mathcal{I}}(x).
\end{align*}
But \([\varphi_{\mathcal{S},\mathcal{S}}(x), c]+\alpha\varphi_{\mathcal{S},\mathcal{I}}(x)\neq 0\), because
 \([\varphi_{\mathcal{S},\mathcal{S}}(x), c]\in \mathcal{N}\) and \(0\neq \varphi_{\mathcal{S},\mathcal{I}}(x)\in \mathcal{I}\).

Thus, we get a contradiction with the supposition on the existence of an automorphism of \(\mathcal{L}\)  which maps \(\mathcal{S}\) onto \(\mathcal{S}_\theta\).
\end{proof}

\begin{exam}\label{E:exam}
Let us consider the following Leibniz algebra  $\mathcal{L}=\left<e_1, e_2, e_3, y_4,y_5, y_6, x_7, x_8, x_9 \right>$ with the
multiplication table:
\[\begin{array}{lll}
\, [e_1, e_2]=-[e_2, e_1]=2e_2, & [e_1, e_3]=-[e_3, e_1]=-2e_3, & [e_2, e_3]=-[e_3, e_2]=e_1, \\
\, [e_1, y_4]=-[y_4, e_1]=y_4,  & [e_1, y_5]=-[y_5, e_1]=-y_5, & \ \\
\, [e_2, y_5]=-[y_5, e_2]=y_4,  & [e_3, y_4]=-[y_4, e_3]=y_5, & \  \\
\, [y_4, y_6]=-[y_6, y_4]=y_4,  & [y_5, y_6]=-[y_6, y_5]=y_5.\\
 \end{array}\]
\[\begin{array}{lll}
\, [x_7, e_1]=-2x_7, & [x_7, e_3]=x_8, & [x_8, e_2]=2x_7, \\
\, [x_8, e_3]=-2x_9,  & [x_9, e_1]=2x_9, & \ \\
\, [x_9, e_2]=-x_8,  & [x_7, y_6]=x_7, & \  \\
\, [x_8, y_6]=x_8,  & [x_9, y_6]=x_9.\\
 \end{array}\]

Then \(\mathcal{L}=\mathcal{S}\dot{+}(\mathcal{M}\dot{+}
\mathcal{I})\), where \(\mathcal{S}=\left< e_1, e_2, e_3 \right> \cong \mathfrak{sl}_2\), $\mathcal{M}=\left< y_4,y_5, y_6 \right>$ and \(\mathcal{I}=\left< x_7, x_8, x_9\right>\).
Since \(\mathcal{N}=\left< x_7, x_8, x_9, y_4, y_5\right> \) and
\([\mathcal{S}, y_6]=\{0\}\), it follows that this Leibniz algebra has the above structure.
\end{exam}

\section{On conjugacy of Levi subalgebras of complex Leibniz algebras} \label{S:complex}

In this section we shall discuss on conjugacy of Levi subalgebras of Leibniz algebras over the field $\mathbb{C}$.

Below, we shall use the following auxiliary lemma.

\begin{lm}\label{degr}
Let \(b \in \mathcal{R}\) be an element such that \([\mathcal{S},
b]\subseteq \mathcal{I}\). Then \(\left. R_b^n\right|_{\mathcal{S}} \in
\Hom_\mathcal{S}(\mathcal{S},\mathcal{I})\) for all \(n\in
\mathbb{N}\).
\end{lm}

\begin{proof} Since \([y,b]\in \mathcal{I}\) and $[\mathcal{L},\mathcal{I}]=\{0\}$, for any \(x, y\in \mathcal{S}\) we have
\[R_b([x, y])=[[x,y],b]=[[x,b],y]+[x, [y,b]]=[[x,b],y]=[R_b(x), y].\]

Now assume that \(R_b^{n-1}\) is an $\mathcal{S}$-module homomorphism, $n \geq 2$. Then
\[
R_b^{n}([x, y])  =
R_b(R_b^{n-1}([x,y]))=R_b([R_b^{n-1}(x),y])=[R_b^n(x), y].
\]
The proof is complete.
\end{proof}

Let $\mathcal{L}=\mathcal{S}\dot + \mathcal{R}$ be a Levi
decomposition of a  Leibniz algebra. Set
\[
\mathcal{E}=\left\{b\in \mathcal{N}: [\mathcal{S}, b]\subseteq
\mathcal{I}\right\}.
\]

\begin{prop} \label{IN} Let \(\mathcal{L}=\mathcal{S}\dot{+}
\mathcal{R}\) be a Levi decomposition of a complex Leibniz algebra
\(\mathcal{L}\) such that  \([\mathcal{S},
\mathcal{E}]=\mathcal{I}\).  Then \([\mathcal{I},
\mathcal{E}]=\{0\}\).
\end{prop}

\begin{proof}
Let us first consider  the case when \(\mathcal{S}\) and
\(\mathcal{I}\) are simple $\mathcal{S}$-modules.

Let us fix a non-zero element \(b\in \mathcal{E}\) such that
\([\mathcal{S}, b]\neq \{0\}\). Taking into account that
\(\mathcal{S}\) and \(\mathcal{I}\) are simple
$\mathcal{S}$-modules, we get  a non-zero module homomorphism
\(\left. R_b\right|_\mathcal{S}\). Applying Schur's lemma we
derive that the modules \(\mathcal{S}\) and  \(\mathcal{I}\) are
isomorphic  and \(\left. R_b\right|_\mathcal{S}\) is a unique up
constant isomorphism.

Let us take  non-trivial \(\tau\in \Hom_\mathcal{S}(\mathcal{S},
\mathcal{I})\). Then by Schur's lemma we have \(\tau=\alpha \left.
R_b\right|_\mathcal{S}, \alpha\in \mathbb{C}\). By
Lemma~\ref{degr}, \(\left. R_b^2 \right|_\mathcal{S}\) is also an
$\mathcal{S}$-module isomorphism and again applying Schur's lemma
we get \(\left. R_b^2 \right|_\mathcal{S}=\lambda \left.
R_b\right|_\mathcal{S}\) for some $\lambda\in \mathbb{C}$.

Since \(b\in \mathcal{N}\), it follows that \(R_b\) is nilpotent,
and therefore there exists $n\in \mathbb{N}$ such that
\(R_b^{n}=0\). Then \(0=\left. R_b^n
\right|_\mathcal{S}=\lambda^{n-1} \left. R_b\right|_\mathcal{S}\),
this implies \(\lambda=0\), because \(\left.
R_b\right|_\mathcal{S} \neq 0\). Thus \(R_b^2=0\). Taking an
arbitrary \(y\in \mathcal{I}\), we see that there exists an
element \(x\in \mathcal{S}\) such that \(R_b(x)=[x, b]=y\). We
have
\[
[y, b]=[[x, b],b]=R_b^2(x)=0.
\]
This means that \([\mathcal{I}, \mathcal{E}]=\{0\}\).

Now let us consider  the case when \(\mathcal{S}\) is semisimple and
\(\mathcal{I}\) is simple $\mathcal{S}$-module.

Let \(\mathcal{S}=\bigoplus\limits_{i=1}^n \mathcal{S}_i\) be  a
decomposition into  sum of  simple Lie ideals. Since
\(\mathcal{I}\) is irreducible, it follows that there exists an
index \(i\) such that \([\mathcal{S}_i,
\mathcal{E}]=\mathcal{I}\). By the above case we obtain that
\([\mathcal{I}, \mathcal{E}]=\{0\}\).

Finally let us consider  the general case.

Let \(\mathcal{I}=\bigoplus\limits_{j=1}^m \mathcal{I}_j\) be  a
decomposition into  sum of  irreducible
\(\mathcal{S}\)-modules. Set
\[
\mathcal{L}_i=\mathcal{S}\dot{+} \mathcal{R}_i,
\]
where \(\mathcal{R}_i=\mathcal{R}/\mathcal{J}_i\),
\(\mathcal{J}_i=\bigoplus\limits_{j\neq i}\mathcal{I}_j\) and
\(i=1, \dots,  m\). Then
\[
[\mathcal{S}, \mathcal{E}_i]=\mathcal{I}_i, \]
for all \(i=1, \dots,  m \), where
\(\mathcal{E}_i=\mathcal{E}/\mathcal{J}_i\). By the above case we
have
\[
[\mathcal{I}_i, \mathcal{E}_i]=\{0\},
\]
for \(i=1,\dots, m\). Thus
\[
[\mathcal{I}_i, \mathcal{E}]=[\mathcal{I}_i, \mathcal{E}_i+
\mathcal{J}_i]= [\mathcal{I}_i, \mathcal{E}_i]=\{0\}.
\]
Further
\[
\left[\mathcal{I},
\mathcal{E}\right]=\left[\bigoplus\limits_{j=1}^m \mathcal{I}_j,
\mathcal{E}\right]=\bigoplus\limits_{j=1}^m \left[\mathcal{I}_j,
\mathcal{E}\right]=\{0\}.
\]
The proof is complete.
\end{proof}

Now we present an example which satisfies the conditions of
Proposition~\ref{IN}.
\begin{exam}
Let us consider the Leibniz algebra \(L(2,0, 1)\) from
\cite{Camacho}:
\[\begin{array}{llll}
\, [e,h] \ =2e,  & [h,f] \ =2f, &[e,f] \ =h, \\
\, [h,e] \ =-2e, & [f,h] \ =-2f, & [f,e] \ =-h,\\
\, [x_1, e]\, =-2x_0,  & [x_2,e] \, =-2x_1, &[x_0,f] \, =x_1, \\
\, [x_1,f]\, =x_2, & [x_0,h] \, =2x_0, & [x_2, h] \, =-2x_2,\\
\, [e, y_1]\ =2x_0, & [f, y_1] \, = x_2, & [h, y_1] \, =2 x_1,\\
\, [y_1,y_2]=y_1, & [y_2,y_1]=-y_1,\\
\, [x_0, y_2]=x_0, & [x_1, y_2]= x_1, & [x_2, y_2]=x_2.\\
 \end{array}\]

In this example  \(\mathcal{L}=\mathcal{S}\dot{+} \mathcal{R}\),
where \(\mathcal{S}=\left<e, f, h\right>\),
\(\mathcal{R}=\left<x_0, x_1, x_2, y_1, y_2\right>\),
\(\mathcal{I}=\left<x_0, x_1, x_2\right>\) and \ \([\mathcal{S},
\mathcal{E}]=\mathcal{I}\), where
\(\mathcal{E}=\left< x_0, x_1, x_2, y_1\right>\).
\end{exam}

Let \(\mathcal{L}=\mathcal{S}\dot{+}\mathcal{R}\) be a Leibniz
algebra and
\begin{equation}\label{sandg}
\mathcal{S}=\mathcal{G}\oplus \mathcal{Q}
\qquad \text{and} \qquad \mathcal{I}=\mathcal{J}\oplus
\mathcal{K}
\end{equation}
be the decompositions from Section~\ref{S:conjugacy}, that is,
\[ \Hom_\mathcal{S}(\mathcal{S},
\mathcal{I})\equiv \Hom_\mathcal{S}(\mathcal{G},
\mathcal{J}) \quad  \text{and} \quad  \Hom_\mathcal{S}(\mathcal{S},
\mathcal{K})=\{0\}.\]

Let  \(\theta\) be an $\mathcal{S}$-module homomorphism such that
\(\theta(\mathcal{G})=\mathcal{J}\). Suppose that
\begin{equation*}
\mathcal{G}=\bigoplus\limits_{i=1}^n \mathcal{G}_i
\end{equation*}
is a decomposition into  sum of  simple Lie ideals. Set

\[\mathcal{J}_i=\theta(\mathcal{G}_i) \ \ \text{for all} \ i\in \{1,\dots,n\}.\]

Since \(\mathcal{G}_i\) is an ideal in \(\mathcal{G}\), it follows
that \(\mathcal{J}_i\) is a submodule of \(\mathcal{J}\).

Take \(z=\theta(x_i)=\theta(x_j)\in \mathcal{J}_i\cap
\mathcal{J}_j\), where \(i\neq j\). Let \(y_k\in \mathcal{G}_k\).

 If \(k\neq i\), then
\[
[z, y_k]=[\theta(x_i), y_k]=\theta([x_i, y_k])=0.
\]

If \(k\neq j\), then
\[
[z, y_k]=[\theta(x_j), y_k]=\theta([x_j, y_k])=0.
\]

Thus \(\left[\mathcal{J}_i\cap \mathcal{J}_j,
\mathcal{G}\right]=\{0\}\), and therefore the intersections
\(\mathcal{J}_i\cap\mathcal{J}_j\) are trivial submodules. Since
\(\theta\) maps \(\mathcal{G}\) onto \(\mathcal{J}\), it follows
that
\begin{equation*}
\mathcal{J}=\bigoplus\limits_{i=1}^n \mathcal{J}_i,
\end{equation*} and each submodule decomposes as
\begin{equation*}
\mathcal{J}_i=\bigoplus\limits_{j=1}^{n_i} \mathcal{J}_j^{(i)},
\end{equation*}
where  $\mathcal{J}_j^{(i)}$ is a simple \(\mathcal{G}_i\)-module.
Since \(\mathcal{J}_i=\theta(\mathcal{G}_i)\), it follows that
$\mathcal{G}_i\simeq \mathcal{J}_j^{(i)}$ for all \(j\in
\{1,\dots, n_i\}\) and \(i\in \{1, \dots,n\}\).

\begin{prop} \label{SRI} Let \(\mathcal{L}=\mathcal{S}\dot{+}
\mathcal{R}\) be a Levi decomposition of a complex Leibniz algebra
\(\mathcal{L}\) such that \([\mathcal{S},
\mathcal{E}]=\mathcal{J}\). Then any two Levi subalgebras of
\(\mathcal{L}\) are conjugate via an inner automorphism.
\end{prop}
\begin{proof}
Let \(\tau\in \Hom_\mathcal{S}(\mathcal{S}, \mathcal{J})\). By
Lemma~\ref{schur} we have  that
\[
\tau=\sum\limits_{i=1}^n\sum\limits_{j=1}^{n_i} \tau_{i j},
\]
where \(\tau_{i j} \colon \mathcal{G}_i\to\mathcal{J}_j^{(i)} \) are
module homomorphisms. For \(i, j\) take an element \(b_{i,j}\in
\mathcal{N}\) such that \(0\neq [\mathcal{G}_i, b_{i,j}]\subseteq
\mathcal{J}_j^{(i)}\). Then there exists \(\lambda_{i j}\in
\mathbb{C}\) such that \(\tau_{i,j}=\lambda_{i,j} R_{b_{i,j}}\).
By Proposition~\ref{IN} we get \( R^2_{b_{i,j}}\big|_\mathcal{S}=0\).

Set
\[
b=\sum\limits_{i=1}^n\sum\limits_{j=1}^{n_i} \lambda_{i,j}b_{i j},
\]
Then
\[
\left. \exp(R_b)\right|_{\mathcal{S}}=\prod\limits_{i=1}^n\prod\limits_{j=1}^{n_i}
\left. \exp(R_{\lambda_{i,j}b_{i
j}})\right|_{\mathcal{S}}=\left. \id\right|_\mathcal{S}+\sum\limits_{i=1}^n\sum\limits_{j=1}^{n_i}
\left. R_{\lambda_{i,j}b_{i j}}\right|_{\mathcal{S}}= \left. \id\right|_\mathcal{S}+\tau.
\]
Due to Corollary~\ref{cormain} the proof of the proposition is complete.
\end{proof}

We give the following criterion on conjugacy of Levi subalgebras.

\begin{thm}\label{T:equiv}
Let $\mathcal{L}$ be a complex Leibniz and let
$\mathcal{L}=\mathcal{S}\dot{+}\mathcal{R}$ be its Levi
decomposition. Then the following assertions are equivalent:
\begin{itemize}
\item[(i)] A Levi subalgebra $\mathcal{S}$ is unique up to conjugacy via an inner automorphism;
\item[(ii)] \([\mathcal{S}, \mathcal{E}]=\mathcal{J}\).
\end{itemize}
\end{thm}

\begin{proof}
(ii) $\Rightarrow$ (i) follows from Proposition~\ref{SRI}.

Let us prove  (i) $\Rightarrow$ (ii). Let
\(\mathcal{S}=\mathcal{G}\oplus \mathcal{Q}\) be a decomposition
of the form \eqref{sandg}.

Let us first consider the case \(\mathcal{G}\) is  a simple Lie
algebra. Let \(\mathcal{J}=\bigoplus\limits_{i=1}^n
\mathcal{J}_i\) be a decomposition into simple modules and
\(\mathcal{G}\cong \mathcal{J}_i\) for all \(i=1, \dots, n\).
Take an arbitrary \(\tau_j \colon \mathcal{G} \rightarrow
\mathcal{J}_j\) module isomorphism. Since any two subalgebras of
\(\mathcal{L}\) are conjugate, there exists an element \(b_j\in
\mathcal{N}\) such that
\[
\left. \exp(R_{b_j})\right|_{\mathcal{G}}=\left. \id \right|_{\mathcal{G}}+\tau_j.
\]
Similarly as in the proof of Proposition~\ref{IN}, we obtain that
\(R_{b_j}^2 \big|_\mathcal{G}=0\). Thus
\(\left.  \tau_j=R_{b_j}\right|_\mathcal{G}\), and therefore \([\mathcal{G},
b_j]=\mathcal{J}_j\). Thus
\[
\mathcal{J}\supseteq [\mathcal{G}, \mathcal{E}]\supseteq
[\mathcal{G}, b_1]+\dots +[\mathcal{G}, b_n]=\mathcal{J}.
\]
Hence \([\mathcal{G}, \mathcal{E}]=\mathcal{J}\).

Now let us consider the general case.

Let \(\mathcal{G}\) be a semisimple Lie algebra and let
\(\mathcal{J}=\bigoplus\limits_{j=1}^m \mathcal{J}_j\) be a
decomposition into sum of  irreducible
\(\mathcal{G}\)-modules. Set
\[
\mathcal{L}_i=\mathcal{S}\dot{+} \mathcal{R}_i,
\]
where \(\mathcal{R}_i=\mathcal{R}/\mathcal{I}_i\),
\(\mathcal{I}_i=\bigoplus\limits_{j\neq i}\mathcal{J}_j\) and
\(i=1, \dots,  m\). Then
\begin{equation}\label{quo}
\Hom_\mathcal{S}(\mathcal{S},
\mathcal{J}_i)\equiv \Hom_\mathcal{S}(\mathcal{G}_i,
\mathcal{J}_i).
\end{equation}
Since a Levi subalgebra $\mathcal{S}$ of \(\mathcal{L}\) is unique
up to conjugacy via an inner automorphism and the quotient of inner
automorphism is also inner, it follows that $\mathcal{S}$ is
unique up to conjugacy via an inner automorphism as Levi
subalgebra of \(\mathcal{L}_i\), for all \(i=1, \dots, m\). Taking
into account \eqref{quo} and   \(\mathcal{G}_i\) is a simple
algebra, by the above case we obtain that
\[
[\mathcal{G}, \mathcal{E}_i]=\mathcal{J}_i,
\]
for all \(i=1, \dots,  m \), where
\(\mathcal{E}_i=\mathcal{E}/\mathcal{J}_i\). Thus
\[
[\mathcal{S}_i, \mathcal{E}]=[\mathcal{S}_i, \mathcal{E}_i+
\mathcal{I}_i]= \mathcal{J}_i.
\]
Further
\[
\left[\mathcal{S},
\mathcal{E}\right]=\left[\bigoplus\limits_{j=1}^m \mathcal{S}_j,
\mathcal{E}\right]=\bigoplus\limits_{j=1}^m \left[\mathcal{S}_j,
\mathcal{E}\right]=\bigoplus\limits_{j=1}^m
\mathcal{J}_j=\mathcal{J}.
\]
The proof is complete.
\end{proof}

\begin{rem}
Note that in Proposition~\ref{exam02} and Example~\ref{E:exam}, we have $\mathcal{E}=0$. Therefore,
$0=[\mathcal{S},\mathcal{E}]\neq \mathcal{J}$.
\end{rem}

\begin{cor} Let \(\mathcal{L}=\mathcal{S}\dot{+}
\mathcal{R}\) be a Levi decomposition of a complex Leibniz algebra
\(\mathcal{L}\) such that \([\mathcal{S},
\mathcal{N}]=\mathcal{J}\). Then any two Levi subalgebras of
\(\mathcal{L}\) are conjugate via an inner automorphism.
\end{cor}

Recall that \(\mathcal{J}\) is the maximal submodule of
\(\mathcal{I}\) such that \(\Hom_\mathcal{S}(\mathcal{S},
\mathcal{I})\equiv \Hom_\mathcal{S}(\mathcal{S},
\mathcal{J})\).

In the following graph we describe the considered cases.

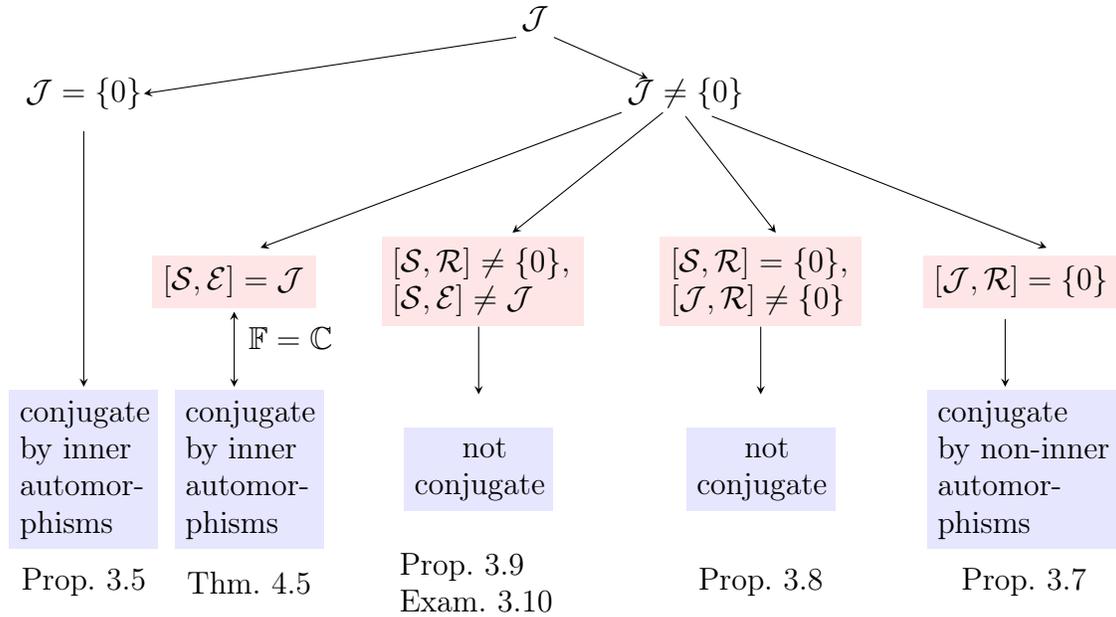
\begin{tikzpicture}
\node at (7,6) {$\mathcal{J}$};
\node at (9,5) {$\mathcal{J} \neq \{0\}$};
\node at (6.3,2.5) [fill=red!10,text width=2.4cm]{$[\mathcal{S}, \mathcal{R}]\neq \{0\}$, $[\mathcal{S}, \mathcal{E}] \neq \mathcal{J}$};
\node at (3,2.5) [fill=red!10,text width=1.9cm] {$[\mathcal{S}, \mathcal{E}] = \mathcal{J}$};
\node at (1,5) {$\mathcal{J}= \{0\} $};
\node at (10,2.5) [fill=red!10,text width=2.4cm] {$[\mathcal{S}, \mathcal{R}] = \{0\}$, $[\mathcal{J}, \mathcal{R}] \neq \{0\} $};
\node at (13.5,2.5) [fill=red!10,text width=2.4cm]{$[\mathcal{J}, \mathcal{R}]=\{0\}  $};
\node at (6.25,0) [fill=blue!10,text width=1.7cm]{\hspace*{0.5cm} not \\ conjugate};;
\node at (3.2,0) [fill=blue!10,text width=1.7cm]{conjugate by inner automorphisms};
\node at (3.75,1.75) {$\mathbb{F}=\mathbb{C}$};
\node at (1,0) [fill=blue!10,text width=1.7cm]{conjugate by inner automorphisms};
\node at (10,0) [fill=blue!10,text width=1.7cm]{\hspace*{0.5cm} not \\ conjugate};
\node at (13.5,0) [fill=blue!10,text width=2.3cm]{conjugate by non-inner automorphisms};
\node at (6.3,-1.5) [text width=2.2cm]{Prop.~\ref{exam02} \\ Exam.~\ref{E:exam}};
\node at (3.2,-1.5) {Thm.~\ref{T:equiv}};
\node at (1,-1.5) {Prop.~\ref{homzero}};
\node at (10,-1.5) {Prop.~\ref{notconj}};
\node at (13.5,-1.5) {Prop.~\ref{conauto}};
\draw [-stealth] (6.75,5.75)--(1.8,5);
\draw [-stealth] (7.25,5.75)--(8.5,5.2);
\draw [-stealth] (9,4.7)--(10.2,3.15);
\draw [-stealth] (8.7,4.75)--(6.7,3.15);
\draw[-stealth] (9.35,4.7)--(13.8,2.95);
\draw[-stealth] (8.15,4.75)--(3.35,2.95);
\draw [-stealth] (6.25,1.9)--(6.25,1);
\draw [-stealth] (10,1.9)--(10,1);
\draw [-stealth] (13.25,2)--(13.25,1.1);
\draw [stealth-stealth] (3,2.1)--(3,1.1);
\draw [-stealth] (1,4.5)--(1,1.1);
\end{tikzpicture}

\section*{Acknowledgements}
This work was supported by Agencia Estatal de Investigaci\'on (Spain), grant MTM2016-79661-P (European FEDER support included, UE), and by Ministry of Education and Science of the Republic of Kazakhstan,
grant No. 0828/GF4.

\end{document}